\documentclass{article}
\usepackage[margin=1in]{geometry}
\usepackage{amsmath,amsthm,amssymb,amsfonts, enumerate, marvosym, tipa, stmaryrd,mathtools,mathrsfs,gensymb,mathtools,mathdots,url,url,hyperref,wasysym}
\usepackage{faktor}
\usepackage{multicol}
\usepackage[all]{xy}
\usepackage{fancyhdr}
\usepackage[style=ieee]{biblatex}

\hypersetup{
    colorlinks=true,
    linkcolor=blue,
    filecolor=magenta,      
    urlcolor=cyan,
}

\newcommand{\N}{\mathbb{N}}

\newcommand{\Q}{\mathbb{Q}}

\newcommand{\B}{\mathcal{B}}
\newcommand{\norm}[1]{\left\vert\left\vert#1\right\vert\right\vert}
\newcommand{\abs}[1]{\left\vert#1\right\vert}

\newcommand{\Zplus}{\mathbb{Z}^{+}}

\DeclareMathOperator{\id}{id}
\DeclareMathOperator{\proj}{proj}
\newcommand{\malg}[1]{\textmd{MALG}(#1)}
\newcommand{\CT}{\textmd{CT}}

\let\oldref\ref
\renewcommand{\ref}[1]{(\oldref{#1})}

\newtheorem{theorem}{Theorem}[section]

\newtheorem{proposition}[theorem]{Proposition}
\newtheorem{corollary}[theorem]{Corollary}
\newtheorem{lemma}[theorem]{Lemma}
\newtheorem{claim}[theorem]{Claim}
\newtheorem{observation}[theorem]{Observation}

\theoremstyle{definition}

\newtheorem{definition}[theorem]{Definition}
\newtheorem{remark}[theorem]{Remark}

\begin{document}
\title{A Backward Ergodic Theorem for Uncountable-to-one Transformations}
\author{Eric Wang}
\date{}
\maketitle

\rfoot{Page \thepage}
\begin{abstract}
    \noindent We establish a generalization of Anush Tserunyan and Jenna Zomback's 2024 Backward Ergodic Theorem. We remove the countable-to-one assumption and thus provide a backward ergodic theorem for arbitrary measure-preserving transformations. However, this new setting introduces measurability concerns as unlike the countable-to-one case, we no longer have a collection of Borel right inverses. Instead, we must rely on the Jankov, von Neumann uniformization theorem. Towards this, we use Borel and measured field structures introduced by Stefaan Vaes and Lise Wouters.
\end{abstract}
\section{Introduction}
In ergodic theory, local-global principles characterize the connection between the orbits of individual points under a measure-preserving transformation and global properties of the transformation. For example, 
Birkhoff's original pointwise ergodic theorem from 1931, appearing in \cite{birkhoffergodic}, states:
\begin{theorem}
    Let $ X $ be a probability space and $ T:X\rightarrow X $ be a measure-preserving transformation. If $ f\in L^1(X,\mu) $, then for almost every $ x\in X $,
    \[
    \sum_{i=0}^N f(T^i(x)) \rightarrow \mathbb{E}(f|\mathcal{B}_\mathcal{I}),\quad N\to \infty 
    \]
    where $ \mathbb{E}(f|\mathcal{B}_{\mathcal{I}}) $ is the conditional expectation of $ f $ with respect to the $ \sigma $-algebra of invariant sets.\\\\
    In particular, if $ T $ is ergodic, then $ \mathbb{E}(f|\mathcal{B}_\mathcal{I}) = \int fd\mu $ a.e.
\end{theorem}
\noindent Essentially, if a system is sufficiently chaotic, more specifically, ergodic, then for almost every point, the forward time average will equal the space average.\\\\
Since then, there have been many other similar theorems that treat more complicated dynamical systems and take more exotic averages over more different sets. For example, in 1987, Grigorchuk, and in 1994, Nevo and Stein in \cite{nevo1994generalization}, independently used weighted averages over spheres to prove an ergodic theorem of measure-preserving actions of free groups. 
\begin{theorem}
    Let $ r\geq 2 $ and let $ \mathbb{F}_r $ be the free group with $ r $ generators acting on the measure space $ (X,\mathcal{B},\mu) $ by measure-preserving transformations. Then for every $ f\in L^1(X,\mu) $, almost every $ x\in X $,
    \[
    \frac{1}{N+1} \sum_{i=0}^N \frac{1}{\abs{S_i}} \sum_{w\in S_i} f(w\cdot x) \to \mathbb{E}(f|\mathcal{B}_{\mathcal{I}}),\quad N\to \infty 
    \]
    where $ \mathcal{E}(f|\mathcal{B}_\mathcal{I}) $ is the conditional expectation of $ f $ with respect to the $ \sigma $-algebra of $ \mathbb{F}_r $-invariant sets, and $ S_N $ denotes the sphere of radius $ n $ centered at the identity of $ \mathbb{F}_r $.
\end{theorem}
\noindent Most recently, in 2024, Tserunyan and Zomback proved pointwise ergodic theorems for measure-preserving actions of free groups and for boundary actions of free groups in \cite{anushpaper}. The backbone for these two theorems was their Backward Ergodic Theorem for countable-to-one measure preserving transformations, which we set up and state below:\\\\
Let $ T $ be an aperiodic countable-to-one pmp Borel transformation on a standard probability space $ (X,\mu) $. Define $ E_T $ to beh orbit equivalence relation of $ T $ - where $ (x,y)\in E_T $ if and only if $ \exists m,n\in \N $ such that $ T^m = T^n y $. Then, let $ (x,y)\mapsto \rho_x(y) $ be the Radon-Nikodym cocucle of $ E_T $. Then:
\begin{theorem}
    \label{countable-to-one}(Tserunyan-Zomback 2024) For every $ f\in L^1(X,\mu) $ and a.e. $ x\in X $, 
    \[
    \frac{\sum_{y\in \tau_x} f(y)\rho_x(y) }{\sum_{y\in \tau_x}\rho_x(y)} \rightarrow \mathbb{E}(f|\mathcal{B}_I) \textmd{ as }\sum_{y\in \tau_x}\rho_x(y)\rightarrow \infty 
    \]
    where $ \tau_x $ ranges over finite-height trees rooted at $ x $. 
\end{theorem}
\noindent Essentially, when averaging over heavier and heavier trees behind $ x $, the average converges to the conditional expectation of $ f $ with respect to the $ \sigma $-algebra of $ T $-invariant sets.\\\\
Their proof is an argument by contradiction. They obtain `bad trees' that witness the failure of the desired properties, and then `glue' the witnesses together into a single tree that occupies most of the complete tree behind $ x $, which will contradict the local-global bridge. There are no serious measurability concerns here, as everything is locally countable, using appropriate uniformization theorems.\\\\
We provide a generalization of their results by removing the countable-to-one assumption and prove an ergodic theorem for arbitrary measure-preserving transformations.
\begin{theorem}
    (Backward Ergodic Theorem for arbitrary measure preserving transformations) Let $ X $ be a standard probability space, let $ T:X\rightarrow X $ be a measure-preserving transformation. Then for every $ f\in L^1(X) $, and for almost every $x \in X $, $ A_{f,E_x} <\infty $ and the limit
    \[
    \lim_{m_x(E_x)\nearrow \infty} A_{f,E_x} 
    \]
    exists and equals $ \overline{f}= \mathbb{E}(f|\mathcal{B}_{\mathcal{I}}) $, the conditional expectation of $ f $ with respect to the $ \sigma $-algebra of $ T $-invariant subsets, and where $ E_x $ ranges over finite-height coherent trees rooted at $ x $.
\end{theorem}
\noindent The style of the proof is similar to the original presentation from Tserunyan and Zomback \cite{anushpaper}, but there are a few key differences. First, we use measure disintegration to play the role of the Radon-Nikodym cocycle. Second, and more importantly, it is more difficult to obtain the invariance and tiling originally discussed. It is not immediate why we can perform this gluing since there may now be uncountably many pieces. However, we will see that this can indeed be done, and we use the theory of Borel fields and fibered spaces discussed in \cite{borelfieldpaper} to justify this.\\\\
We also remove the countable-to-one assumption for their Backward Maximal Ergodic Theorem:
\begin{theorem}
    Let $ X $ and $ T $ be as above. Then for every $ f\in L^1(X,\mu) $, if we define $ f^\star  = \sup_{E_x\in \CT_x} A_{f,E_x} $ for each $ x\in X $. Then for any $ \lambda\in \mathbb{R} $, 
    \[
    \int_{\{f^\star>\lambda\}} f d\mu \geq \lambda \mu(\{f^\star>\lambda\})
    \]
\end{theorem}
\noindent The proof technique mimics the presentation from Tserunyan and Zomback. In their proof, they fix minimal (with respect to proper sub-trees) witnesses $ \tau_x $ for $ x $ for which $ f^\star(x)>\lambda $. The important observation for such witnesses is that any other $ y\in \tau_x $ must also have the property that $ f^\star(y)>\lambda $. They then use the map $ x\mapsto \tau_x $ for their tiling property and then the rest of the proof proceeds identically to their proof of their Backward Ergodic Theorem. \\\\
In our more general context, while we can prove that analogous minimal witnesses exist, we cannot necessarily fix such a mapping $ x\mapsto E_x $ of witnesses in a measurable way. However, we use the existence of such minimal witnesses to prove that we can measurably fix witnesses that satisfy the weaker condition of for almost-every $ y\in E_x $, $ f^\star(y)>\lambda $. After this, we can similarly use the same techniques from our general Backward Ergodic Theorem to finish the proof.

\section{Background and Preliminaries}
\noindent We define a \textbf{standard probability space} to be a triple $ (X,\mathcal{B},\mu) $ where $ (X,\mathcal{B}) $ is standard Borel and $ \mu $ is a complete Borel probability measure on $ X $. While we do require $ \mu $ to be complete, it will still be necessary to consider the original Borel sets, $ \mathcal{B} $. \\\\
A Borel measurable transformation, $ T:X\rightarrow X $ is ($\mu$)-\textbf{measure preserving} if for every $ A\in \mathcal{B} $, $ \mu(A) = \mu(T^{-1}[A]) $. We say a set $ A\in \mathcal{B} $ is ($T$)-\textbf{invariant} if $ \mu(A\triangle T^{-1}[A]) = 0 $. Furthermore, for $ n\geq 1 $, we use $ T^n:X\rightarrow X $ to denote the $ n $-fold composition of $ T $. So,
\[
T^n(x) = T\circ T\circ\ldots \circ T(x)
\]
where there are $ n $-many $ T $'s being applied.
\subsection{Measure Disintegration and Conditional Expectation} 
\begin{definition}
    Let $ (X,\mathcal{B},\mu) $ be a measure space and suppose $ \mathcal{D} $ is a sub-$\sigma$-algebra of $ \mathcal{B} $. If $ f\in L^2(X,\mathcal{B},\mu) $, then the \textbf{conditional expectation of $ f $ with respect to $ \mathcal{D} $} is the orthogonal projection of $ f $ to the subspace of $ L^2(X,\mathcal{D},\mu) $, and is denoted
    \[
    \mathbb{E}(f|\mathcal{D})
    \]
    The function $ \mathbb{E}(f|\mathcal{D}) $ is the unique, up to equality off a null set, $ \mathcal{D} $-measurable function for which for every $ A\in \mathcal{D} $,
    \[
    \int_A fd\mu = \int_A \mathbb{E}(f|\mathcal{D})d\mu 
    \]
    Note that conditional expectation can also be generalized to $ L^1 $ functions.
\end{definition}
\noindent We also rely heavily on the measure disintegration theorem, as discussed in \cite{Fremlin2013}, [452G], [452I], [452XG], and [452XL]. We provide the form we use below.
\begin{theorem}\label{disintegration}
    (Measure Disintegration) Let $ X $ be a standard probability space and let $ T:X\rightarrow X $ be a (Borel) measure-preserving transformation. Then there exists a family of (complete) probability measures, $ \{\mu_x:x\in X\} $, each supported on $ T^{-1}\{x\} $ and defined on all of $ \mathcal{B}$, so that for every Borel set $ E $, the function $ x\mapsto \mu_x(E) $ is integrable and
    \[
    \int \mu_x(E)d\mu(x) = \mu(E)
    \]
    Furthermore, if $ \{\nu^x:x\in \mathbb{X}\} $ is another family of probability measures for which $ \nu_x $ is supported on $ T^{-1}\{x\} $ and $ \int \nu_x(E)d\mu(x) = \mu(E) $ for every Borel set $ E $, then for almost every $ x $, $ \mu_x = \nu_x $. 
\end{theorem}
\noindent As an immediate observation, given a Borel set $ E $, the map $ X\rightarrow \mathbb{R}, x\mapsto \mu_x(E) $, is Borel measurable almost everywhere - that is, there is a Borel measurable function $ \eta $ that agrees with $ x\mapsto \mu_x(E) $ almost everywhere.\\\\
Furthermore, if $ f\in L^1(X,\mathcal{B},\mu) $, then for almost every $ x\in X $, $ f\in L^1(X,\mathcal{B},\mu_x) $, and the function $ x\mapsto\int f(y)d\mu_x(y) $ is integrable with
\begin{equation}
    \int \int f(y)d\mu_x(y)d\mu(x) = \int f(x)d\mu(x)
\end{equation}
\noindent In addition, if $ \mu $ in $ (X,\mathcal{B},\mu) $ is not complete, there is still a disintegration, and the completion of the measures that form the disintegration will be the disintegration of the completion of $ \mu $.
\subsection{Uniformization}
\begin{definition}
    If $ R\subseteq X\times Y $ is a relation. Then a \textbf{uniformization} of $ R $ is a function $ f:\proj_X(R) \rightarrow Y $ such that $ (x,f(x))\in R $ for each $ x\in \proj_X(R) $. 
\end{definition}
\noindent The relevant uniformization theorem used in the countable-to-one case was the Lusin-Novikov uniformization theorem. We will instead use the Jankov, von Neumann uniformization theorem in \cite{Kechris95} [18.1]
\begin{theorem}
    \label{JVN} (Jankov, von Neuman) Let $ X $ and $ Y $ be standard Borel. If $ P\subseteq X\times Y $ is analytic. Then there exists a $ \sigma(\Sigma^1_1) $-measurable uniformization of $ P $.
\end{theorem}
\noindent Furthermore, any $ \sigma(\Sigma^1_1) $-measurable set is universally measurable - that is, it is measurable with respect to any complete Borel probability measure (see \cite{Kechris95} [21.10]).
\subsection{Relevant Definitions}
Let $ (X,\mathcal{B},\mu) $ be a standard measure space, and $ T:X\rightarrow X $ a Borel measure-preserving transformation. We make no assumptions on the cardinality of $ T^{-1}\{x\} $.
\begin{definition}
    For each $ i\in \Zplus $, $ T^i $ is a Borel measure-preserving transformation. We apply the Measure Disintegration theorem \ref{disintegration} to $ T^i $, and we let $ \{\mu_x^i:x\in X\} $ be the resulting family of probability measures that satisfy the properties described in the theorem. \\\\
    As such, for each $ x $, we let $ \mu_x^i $ be \textbf{the disintegration of $ \mu $ over the measure-preserving transformation $ T^i $ at $ x $}.
\end{definition}
\noindent For convenience of notation, we let $ \mu^0_x $ be $ \delta_x $ - the dirac mass at $ x $. We have the following key observation:
\begin{proposition}
    For almost every $ x\in X $ and every $ A\in \mathcal{B} $, 
    \begin{equation}
        \label{eq:1}
    \mu^2_x (A) = \int_X \mu^1_y (A)d\mu^1_x(y)
    \end{equation}
\end{proposition}
\begin{proof}
    It's clear that the right-hand-side is indeed a measure - denote it $ \nu_x $. To show it equals $ \mu^2_x $, by uniqueness of measure disintegration, it suffices to show that it is also a disintegration of $ T^2 $.\\\\
    We must how two things - first that $ \nu_x $ is supported on $ T^{-2}\{x\} $, and second, for each $ E\in \mathcal{B} $, that $ x\mapsto \nu_x(E) $ is integrable and
    \[
    \int \nu_x(E) d\mu(x) = \mu(E)
    \]
    For the first, Let $ B = X\setminus T^{-2}\{x\} $. Then
    \[
    \nu_x(B) = \int_X \mu^1_y (B) d\mu^1_x (y) = 0
    \]
    since $ \mu^1_x $ is supported on $ T^{-1}\{x\} $, $ y\in T^{-1}\{x\} $, and $ \mu^1_y $ is supported on $ T^{-1}\{y\}\subseteq T^{-2}\{x\} $.\\\\
    For the second, let $ E\in \mathcal{B} $. Since $ \mu_y^1 $ are measure disintegrations, the function $ y\mapsto \mu_y^1(E) $ is integrable. As such, by the disintegration theorem, the function $ x\mapsto \int_X \mu_y^1(E)d\mu_x^1(y) = \nu_x(E) $ is also integrable, and we see
    \begin{align*}
        \int_X \nu_x(E)d\mu(x) &= \int_X \int_X \mu^1_y(E)d\mu^1_x(y)d\mu(x)\\
        &= \int_X\mu_x^1(E)d\mu(x)\\
        &= \mu(E)
    \end{align*}
    as desired.
\end{proof}
\noindent By an inductive argument, we obtain
\begin{proposition}\label{coherencelemma}
    For almost every $ x\in X $ and every $ A\in \mathcal{B} $, and for every $ i\geq 1 $,
    \[
    \mu^{i+1}_x(A) = \int_X \mu^i_y(A)d\mu^1_x(y)
    \]
\end{proposition}
\begin{proof}
    The prior proposition is the base case. Suppose it is true for $ i\geq 1 $. We will show
    \[
    \mu^{i+2}_x(A) = \int_X \mu^{i+1}_y(A)d\mu^1_x(y)
    \]
    by again showing the right-hand-side is a disintegration for $ T^{i+2} $. Let $ \nu_x $ denote the right-hand-side. \\\\
    An identical argument as above shows that $\nu_x$ is supported on $ T^{-(i+2)}\{x\} $. \\\\
    Furthermore, if $ E\in \mathcal{B} $, the function $ x\mapsto \mu_x^{i+1}(E) $ is again integrable. Hence,
    \begin{align*}
        \int_X \nu_x(E) d\mu(x) &= \int_X \int_X \mu_y^{i+1}(E)d\mu_x^1 d\mu(x) \\
        &= \int_X \mu_x^{i+1}(E)d\mu(x)\\
        &= \mu(E)
    \end{align*}
    as desired
\end{proof}
\noindent As an immediate consequence, we see
\begin{equation}
    \mu^i_x(A) = \int \int \ldots \int \mu^1_{y_{i-1}}(A)d\mu^1_{y_{i-2}}(y_i)\ldots d\mu^1_{x}(y_1)
\end{equation}
\begin{definition}
    For each $ x\in X $, we consider the set of \textbf{coherent trees rooted at $ x $}, denoted by $ \CT_x $. More specifically, we consider sequences $ E_x = (E^0, E^1,\ldots, E^n) $ where the following are true:
\begin{enumerate}
    \item $ E^0 = \{x\} $, and for each $ i\leq n$, $ E_i \in \malg{\mu_x^i} $ - each level is measurable.
    \item For each $ i \leq n $, $ \mu_x^i(E^i) > 0 $ - each level has positive mass
    \item For $ 0\leq i < n $, $ \mu_x^{i+1}(E^{i+1}\setminus T^{-1}[E^i]) = 0 $ - the levels are coherent; essentially, `$ E^{i+1}\subseteq T^{-1}[E^i] $'.
\end{enumerate}
\end{definition}
\noindent For the above $ E_x $, we say the \textbf{length} of $ E_x $, $ \ell(E_x)$, is $ n $. Then, given a tree $ E_x $, say of length $ n $, we define the \textbf{mass of $ E_x $ at $ x $} as
\[
m_x(E_x):= \sum_{i=0}^n \mu^i_x(E^i)
\]
As an abuse of notation, given some tree $ E_x $ of length $ n $, we may go between $ E_x $ and the disjoint union of its levels, $ \bigsqcup_{i=0}^n E^i $ to refer to it. As such, we write
\[
\int_{\bigsqcup_{i=0}^n E^i} fd\mu = \sum_{i=0}^n \int_{E^i} f d\mu_x^i
\]
\begin{definition}
    We let $ \triangleright_T^N\cdot x $ denote the \textbf{complete tree} of length $ N $. Observe that
    \[
    m_x(\triangleright_T^N\cdot x) = N+1 
    \]
\end{definition}
\noindent Then given $ f\in L^1(X,\mathcal{B},\mu) $, we can define the \textbf{average of $ f $ over $ E_x $} as
\[
A_{f,E_x} = \frac{\sum_{i=0}^{\ell(E_x)} \int_{E^i}f(z)d\mu^i_x(z)}{m_x(E_x)}
\]
Then, we let
\[
f^*(x) = \limsup_{m_x(E_x)\nearrow \infty} A_{f,E_x},\quad f_*(x) = \liminf_{m_x(E_x)\nearrow \infty} A_{f,E_x}
\]
where $ E_x $ varies over $ \CT_x $.\\\\
Essentially, $ f^* $ and $ f_* $ are the upper and lower estimates of the average value of $ f $ behind $ x $.
\section{Borel Fields and Measurability}
As mentioned in the previous section, we consider the \textbf{ coherent trees rooted at $  x$}. Our goal is, given a point $ x $ and many trees rooted in its preimage, $ T^{-1}\{x\} $, to be able to amalgamate potentially uncountably many trees into a single tree rooted at $ x $. To demonstrate that this can be done in a measurable way while keeping track of our many disintegrations, we will appeal to the fibered and measured Borel field structure discussed in \cite{borelfieldpaper}. Again, we work with a fixed integrable function $ f $ on $ X$.
\begin{definition}
    Let $ (X,\mathcal{B},\mu) $ be a measure space. The \textbf{measure algebra for $ \mu $}, $\malg{\mu} $, is a metric space consisting of $ \mathcal{B}/\sim_\mu $ (equivalence classes mod null) with metric
    \[
    d(A,B) = \mu(A\triangle B)
    \]
\end{definition}
\begin{lemma}\label{uniformly dense}
    Let $ (X,\mathcal{B}) $ be standard Borel. Then there exists $ \mathcal{C}\subseteq \mathcal{B} $ countable such that for every Borel probability measure $ \mu $ on $ X $, $ \mathcal{C} $ is dense in $ \malg{\mu} $. 
\end{lemma}
\begin{proof}
    As $ X $ is standard Borel, we have a compatible Polish topology. In particular, let $ \mathcal{U} $ be a countable base for $ X $, and let $ \mathcal{C} $ be the collection of all finite unions of $ \mathcal{U} $. We claim that $ \mathcal{C} $ is the witnessing countable collection. It is clearly countable. Let $ \mu $ be a Borel probability measure on $ X $. Then, as $ X $ is Polish, $ \mu $ is Radon. In particular, $ \mu $ is outer regular, that is, for every $ S\in \mathcal{B} $,
    \[
    \mu(S) = \inf\{\mu(U):S\subseteq U\}
    \]
    Let $ \varepsilon > 0 $. Then let $ U $ be open such that $ S\subseteq U $ and $ \mu(U\setminus S)<\frac{\varepsilon}{2} $. As $ U $ is open, it is a countable union of elements from $ \mathcal{U} $. In particular, by the continuity of measures from below, there exists $ C\in \mathcal{C} $ such that $ C\subseteq U $ and
    \[
    \mu(U\setminus C) < \frac{\varepsilon}{2}
    \]
    We then have
    \[
    \mu(S\triangle C)<\varepsilon
    \]
    as desired.
\end{proof}
\noindent We will call such a family obtained in \ref{uniformly dense}, \textbf{uniformly dense}\\\\
\noindent By measure disintegration, for each $ n,m \in \mathbb{N} $, and  every $ x $, the function $ x\mapsto \mu_x^n(C_m) $ is Borel. 
\noindent We will now place a standard Borel structure on 
\[
V_i:=\bigsqcup_{x\in X} \malg{\mu_x^i} 
\]
Toward this endeavor, we will use the Borel field and the measured field structure introduced in \cite{borelfieldpaper}. We will reuse their notation of
\[
V_i\times_\pi V_i = \{(x,y)\in V_i\times V_i: \pi(x)=\pi(y) \}
\]
\begin{definition}
    (Definition 2.1 in \cite{borelfieldpaper}) A \textbf{Borel field of Polish spaces} is of a standard Borel space $ X $, a family $ V = (V_x)_{x\in X} $ of Polish spaces, and a standard Borel structure on $ V = \bigsqcup_{x\in X}V_x $ such that the following holds: \begin{enumerate}
        \item The map $\pi:V\rightarrow X $ where $ v\in X_{\pi(v)} $ is Borel.
        \item There exists a Borel map $ d:V\times_\pi V \rightarrow [0,\infty) $ such that for every $ x\in X $, the restriction of $ d $ to $ V_x\times V_x $ is a compatible complete metric for $ V_x $.
        \item There exists a sequence of Borel functions $ \varphi_n:X\rightarrow V $ such that for all $ x\in X $, $ \varphi_n(x)\in V_x $ and $ \{\varphi_n(x):n\in \mathbb{N} \} $ is dense in $ V_x $ with respect to $ d|_{V_x\times V_x} $. 
    \end{enumerate}
\end{definition}
\begin{definition}
    (Definition 2.4 in \cite{borelfieldpaper}) A \textbf{measured field of Polish spaces} is a standard $ \sigma $-finite measure space $ (X,\mu) $, a family $ V = (V_x)_{x\in X} $ of Polish spaces and a standard Borel structure on $ V $ for which the following hold:
    \begin{enumerate}
        \item $ \pi:V\rightarrow X $ is Borel
        \item There is a conull Borel set $ X_0\subseteq X $ and such that the restriction $ (V_x)_{x\in X_0} = \pi^{-1}(X_0) $ is a Borel field of Polish spaces. 
    \end{enumerate}
\end{definition}
\noindent In summary, given a measure on a standard Borel space, we can ignore a measure $ 0 $ set of points when we make our fibered Borel space.\\\\
\noindent As such, before we continue, we give the conull Borel set on which we will have a Borel field structure.
\begin{remark}\label{universeofinterest}
    There exists a $ \mu $-conull Borel set, $ X' $ such that the following properties hold:
    \begin{enumerate}
        \item For every $ x\in X' $, and every $ A\in \mathcal{B} $, and for all $ n\geq 1 $,
        \[
        \mu_x^{n+1}(A) = \int_X \mu_y^n(A)d\mu_x^1(y)
        \]
        \item From \ref{uniformly dense}, let $ C' $ is the algebra generated by $ \bigcup_{i\in \mathbb{N}}\{T^{-i}(C_k):k\in \mathbb{N}\} $. Then for every $ x\in X'$, for every $ s\in \mathcal{C'} $, and every $i\in \mathbb{N} $, the function
        \[
        \eta_{i,s}:x\mapsto \mu_x^i(s)
        \]
        is Borel measurable.
        \item The $ f\in L^1 $ from before has $ f|_{X'} $ is Borel.
        \item For every $ x\in X' $, for every $ n\in \mathbb{N} $, $ T^n(x)\in X' $.
    \end{enumerate}
\end{remark}
    \begin{proof}
        For the first condition, we already established this for $ \mu $-almost every $ x\in X $. So, we let $ X_1 $ be a $ \mu $-conull Borel set on which this holds.\\\\
        For the second condition, there are only countably many such $ \eta $'s that we require to be measurable, so we let $ X_2 $ be a $ \mu $-conull Borel set such that on $ X_2 $, $ \eta_{i,s}$'s are all simultaneously Borel measurable.\\\\
        For the third condition, we let $ X_3 $ be Borel and $ \mu $-conull on which $ f|_{X_3} $ is Borel measurable. \\\\
        Now, let $ X_0 = X_1\cap X_2\cap X_3 $, this is still $ \mu $-conull and Borel. And, the three properties described pass down to $ X_0 $. \\\\
        Finally, since $ T $ is measure preserving, $ X':= \bigcap_{n\in \mathbb{N}}T^n(X_0)\subseteq X_0 $ is $ \mu $-conull, and properties $ (1) $, $ (2) $, and $ (3) $ still hold, and we see that $ X' $ satisfies $ (4) $.
    \end{proof}
\noindent As described in Proposition 4.2 of \cite{borelfieldpaper}, we can induce a Borel field structure on a family of Polish spaces by providing adequate functions that we insist must be Borel measurable
\begin{proposition}\label{prop 4.2}
    (Prop 4.2 in \cite{borelfieldpaper}) If $ X $ is standard Borel, $ V = (V_x)_{x\in X} $ a family of Polish spaces with compatible complete metrics $ d_x $ for each $ x\in X $ on $ V_x $, and $ \mathcal{S} $ is a collection of sections $ \varphi:X\rightarrow V $ such that the following are true:
    \begin{enumerate}
        \item For all $ \varphi,\psi\in \mathcal{S} $, $ x\mapsto d_x(\varphi(x),\psi(x)) $ is Borel as a function from $ X $ to $ \mathbb{R} $
        \item The family $ \mathcal{S} $ is maximal - that is, if $ \psi:X\rightarrow V $ is a section such that for each $ \varphi\in \mathcal{S} $, $ x\mapsto d_x(\varphi(x),\psi(x)) $ is Borel, then $ \psi\in \mathcal{S }$. 
        \item There exists sequence $ \varphi_n $ in $ \mathcal{S} $ such that for each $ x\in X $, the set $ \{\varphi_n(x):n\in \N\} $ is dense in $ V_x $ with respect to $ d_x $. 
    \end{enumerate}
    Then there is a unique standard Borel structure on $ V $ such that $ \pi:V\rightarrow X$, $ d:V\times_\pi V \rightarrow[0,\infty) $, and $ \varphi:X\rightarrow V $ for $ \varphi\in S $ are all Borel. 
\end{proposition}
\begin{proposition}\label{borelonmalg}
    There exists a measured field structure structure on $ M:=\bigsqcup_{x\in X} \malg{\mu_x^1} $. Namely, there is a conull Borel set $ X_0\subseteq X $ such that the following are true on $ \bigsqcup_{x\in X_{0}}\malg{\mu_x^1} $:
    \begin{itemize}
        \item $ \pi: M\rightarrow X $ where $ \pi(m) = x $ where $ m\in \malg{\mu_x^1} $ is Borel measurable
        \item $ d:M\times_\pi M \rightarrow [0,\infty) $ where 
        \[
        d(a,b) = \mu_x^1(a\triangle b) 
        \]
        is Borel measurable
        \item There exists a sequence of Borel functions $ \{\varphi_n:X\rightarrow M \} $ such that for each $ x\in X $, $ \{\varphi_n(x):n\in \N\} $ is dense in $ \malg{\mu_x^1} $ with respect to metric $ d_x $. 
    \end{itemize} 
\end{proposition}
\begin{proof}
    The conull set is $ X_0 $ described in \ref{universeofinterest}, and we apply \ref{prop 4.2} to obtain the desired standard Borel structure.\\\\
    Enumerate the countable uniformly dense $ \mathcal{C} = \{C_i:i\in \mathbb{N}\} $ from Lemma \ref{uniformly dense}. Define Borel sections $ \varphi_n $ via
    \[
    \varphi_n(x) = \begin{cases}
        [C_{n-1}]_x & n\geq 1\\
        [\emptyset]_x & n = 0
    \end{cases}
    \]
    Then if $ n,m\in \omega $, then
    \begin{align*}
        d_x(\varphi_n(x),\varphi_m(x)) &= \mu^i_x ([C_n\triangle C_m]_x)\\
        &= \mu^i_x(C_n\triangle C_m)
    \end{align*}
    Thus, by \ref{universeofinterest}, $ x\mapsto d_x(\varphi_n(x),\varphi_m(x)) $ is Borel.\\\\
    Then, a standard Zorn's lemma argument yields (2) in \ref{prop 4.2}.
\end{proof}
\begin{corollary}
    As with the $ k=1 $ case, there exists a standard Borel structure on $ M:=\bigsqcup_{x\in X_{0}} \malg{\mu^i_x} $.
\end{corollary}
\begin{proof}
    We repeat the same argument as above, but also insist that for each $ i\leq k $, that the following sections are measurable:
    \[
    \varphi_{i,n}:x\mapsto [T^{-i}(C_n)]_{\mu_x^k}
    \]
    There are only countably many additional functions to insist are Borel and it is compatible with condition (1) of Proposition 4.2 of \cite{borelfieldpaper}.
\end{proof}
\begin{remark}
    By neglecting the null set $X\setminus X_0 $, we may assume that the Borel field structure is defined on all of $ X $. To simplify notation, we will just write $ X $. 
\end{remark}
\begin{corollary}\label{cor1}
    For any $ A \in V $, the following function is Borel:
    \[
    A\mapsto \mu^i_{\pi(A)}(A)
    \]
\end{corollary}
\begin{proof}
    This follows from Lemma 2.3(1) of \cite{borelfieldpaper} since 
    \[
    d(A,\varphi_0(\pi(A))) = d(A,[\emptyset]_{\pi(A)}) = \mu^i_{\pi(A)}(A) 
    \]
\end{proof}
\begin{definition}
     For each $ i $, we have a standard Borel structure on $ V_i:= \bigsqcup_{x\in X} \malg{\mu^x_i} $. Let $ \mathcal{V}_n = X \times \prod_{i=1}^n V_i $, which is also standard Borel with the product Borel $ \sigma $-algebra. Finally, let $ \mathcal{Y} = \bigsqcup_{n<\omega} \mathcal{V}_n $, which is also standard Borel as a disjoint union of countably many standard Borel spaces
\end{definition}
\begin{observation}
    For each $ i\in \mathbb{N} $, we let $ \proj_i:\mathcal{Y}\rightarrow V_i $ denote the projection onto the $ i $'th coordinate (where it makes sense),
    \[
    \proj_i:E\mapsto E^i
    \]
    Then $ \proj_i $ is a Borel measurable function. For ease of notation, we will use $ E^i $ to refer to $ \proj_i(E) $.
\end{observation}
\begin{lemma}\label{first lem}
    Let $ T^{-1}:V_i\rightarrow V_{i+1} $ be $ A\mapsto T^{-1}[A] $. This is Borel between Borel fields. 
\end{lemma}
\begin{proof}
    First, since $ T $ is measure-preserving, it is well defined. We will apply Lemma 2.6 (2) from \cite{borelfieldpaper}. For each $ x\in X $, $ T $ is continuous as a map between $ \malg{\mu_x^i} $ and $ \malg{\mu_x^{i+1}} $ as it is an isometry due to measure disintegration. Using the $ \varphi_n $ obtained in \ref{borelonmalg}, it remains to check that $ T\circ \varphi_n:X\rightarrow V_{i+1} $ is Borel for every $ k $. Fix $ k\in \N $. Note that for every $ x $, $ T\circ \varphi_k(x) = [T^{-1}[C_k]]_{\mu_x^{i+1}} $. But this is precisely the function $ \varphi_{k,i+1} $, which is Borel for free.
\end{proof}
\begin{lemma}\label{second lem}
    Let $ W_i = V_i\times_\pi V_i $. Then $ W_i $ also has a Borel field structure. Furthermore, for each operation $ \cdot \in \{\cup, \cap, \setminus\} $, the function $ \zeta: W_i \rightarrow V_{i}$ given by $ (A,B)\mapsto A \cdot B $, is Borel between Borel fields. 
\end{lemma}
\begin{proof}
    Since for each $ x $ and $ i $, $ \malg{\mu_x^i} $ is Polish, so is $ \malg{\mu_x^i}\times \malg{\mu_x^i} $; use the metric 
    \[
    d((A_1,A_2),(B_1,B_2)) = \mu_x^i(A_1\triangle B_1)+\mu_x^i(A_2\triangle B_2) 
    \] 
    Then we once again use \ref{prop 4.2} to obtain a standard Borel structure on $ W_i $. From \ref{uniformly dense}, $ \mathcal{C}\times \mathcal{C} $ is also uniformly dense in $ \malg{\mu_x^i}\times \malg{\mu_x^i} $ for every $ x $, and is still countable. We let $ S = \{\psi_{m,n}:m,n\in \mathbb{N}\} $ where
    \[     
    \psi_{i,j}(x) = (\varphi_i(x),\varphi_j(x))     
    \] 
    where $ \varphi_n $ is given in \ref{borelonmalg}. Then, using an identical argument as in \ref{borelonmalg}, we obtain the desired standard Borel structure. \\\\
    For the second part, as with the previous lemma, we use Lemma 2.6 (2) from \cite{borelfieldpaper}. For each $ x \in X_0 $, $ \zeta $ is continuous as a map between $ \malg{\mu_x^i}\times \malg{\mu_x^i} $ and $ \malg{\mu_x^{i}} $. And, using the $ \psi $ obtained in the first part of this proof, we must check that $ \zeta\circ \psi: X\rightarrow V_i $ is Borel. But this is immediate from the proof of \ref{borelonmalg} and \ref{universeofinterest}.
\end{proof}
\noindent A similar argument shows the following:
\begin{lemma}\label{third lem}
    Let $ W_{i,j} = V_i\times_\pi V_j $. Then $ W_{i,j} $ also has a Borel field structure. Furthermore, the function $ \eta:W_{i+1,i}\rightarrow W_i $, $ (A,B)\mapsto (A,T^{-1}(B)) $ is Borel between Borel fields.
\end{lemma}
\begin{lemma}\label{BorelCT}
    The following functions and sets are Borel measurable:
    \begin{enumerate}
        \item The function that computes the mass of an element of $ \mathcal{Y} $:  $E\mapsto m_{\pi(E)}(E) $
        \item For a fixed $ f\in L^1 $, $ E \mapsto \sum_{i=0}^{\ell(E)} \int_{E^i}fd\mu_i^{\pi(E)}$
        \item The set $ \CT $ of coherent trees, as well as $ \CT_x $ for each $ x $.
        \item The function $ E\mapsto A_{f,E} $
    \end{enumerate}
\end{lemma}
\begin{proof}
    \begin{enumerate}
        \item This is a sum of Borel functions: 
        \[
        \sum_{i=0}^{\ell(E)} \mu_{\pi(E^i)}^{i}(E^i) 
        \]
        \item On $ X $, $ f $ is Borel measurable. Fix a sequence of Borel measurable simple functions $ s_n $ such that $ \abs{s_n}\leq \abs{f} $ and $ s_n $ converges pointwise to $ f $ everywhere (on $ X $). For each $ n $, let $ \gamma_n:\mathcal{Y}\rightarrow \mathbb{R} $, where
        \[
        \gamma_n(E) = \sum_{i=0}^{\ell(E)} \int_{E^i} s_n d\mu_{\pi_0(E)}^i  
        \]
        Then $ \gamma_n $ is Borel measurable as it is a finite linear combination of functions of the form from \ref{cor1}. Then, $ E\mapsto \sum_{i=0}^{\ell(E)}\int_{E^i} f\mu_{\pi_0(E)}^i $ 
        is the pointwise limit of $ \gamma_n $ by the Dominated Convergence Theorem. Hence it is Borel measurable
        \item We observe that $ E $ is a coherent tree if and only if the following are true:
        \begin{itemize}
            \item For all $ i\leq \ell(E) $, $ \pi(E^i) = x $. 
            \item For all $ i \leq \ell(E) $, 
            \[
            d^i_x(E^{i+1}\setminus T^{-1}[E^i],\varphi_0(\pi(E^i))) = 0 
            \]
        \end{itemize}
        It is easy to see that the first calculation is Borel. The second is also Borel from \ref{first lem}, \ref{second lem}, and \ref{third lem}.
        \item This follows from (1) and (2).
    \end{enumerate}
\end{proof}
\begin{lemma}
    \label{gluing}(Gluing Lemma) Let $ x\in X $ and $ n\in \N $, and suppose $ A \subseteq T^{-n}\{x\} $ be measurable. Further suppose that for each $ y\in A $, that the function $ y\mapsto E_y $ is measurable as a function from $ X $ into $ \mathcal{Y} $ with $ E_y\subseteq T^{-1}\{y\} $. Then there exists a measurable set $ F\subseteq T^{-(n+1)}\{x\} $ so that for $ \mu_x^n $-a.e. $ y\in A $, $ \mu_y^1(F) = \mu_y^1(E_y) $. Furthermore, if $ f\in L^1 $, then 
    \[
    \int_{F} fd\mu_x^{n+1} = \int_A \int_{E_y} fd\mu_y^1d\mu_x^n(y)
    \].
\end{lemma}
\begin{proof}
    We recall $ \mathcal{C} = \{C_i:i<\omega\} $ from \ref{uniformly dense}. By the previous lemma, the map $ y\mapsto \mu_y^1(E_y) $ is a Borel measurable function from $ X $ into $ \mathbb{R} $. For each $ m\in \N $ and $ \varepsilon>0 $, let
    \[
    D^\varepsilon_m = \{y\in A:\mu_y^1(E_y\triangle C_m)<\varepsilon\}
    \]
    This is the set of $ y\in A $ for which $ C_m $ closely approximates $ E_y $. Then $ D^\varepsilon_m $ is measurable. Let $ F^\varepsilon_1 = D^\varepsilon_1 $, and for $ k\geq 1 $, let $ F_{k+1} = D^\varepsilon_{k+1}\setminus F_k $. Then $ F_k $ are all Borel and are disjoint. We see $ F_k $ are the set of $ y\in A $ for which the first member of $ \mathcal{C} $ which closely approxmiates $ E_y $ is $ C_k $. Then let 
    \[
    F_\varepsilon:= \bigcup_{k<\omega} T^{-1}[F_k^\varepsilon]\cap C_k
    \]
    which is also measurable.\\\\
    Now, let $ \varepsilon_m \searrow 0 $, and $ F_{\varepsilon_m} $ is Cauchy in $ \malg{\mu_x^{n+1}} $. More specifically, $ \{F_{\varepsilon_m}:m<\omega\} $ forms a Cauchy sequence in $ \malg{\mu_x^{n+1}} $. Indeed, we see
    \begin{align*}
        \mu_x^{n+1}(F_{\varepsilon_m}\triangle F_{\varepsilon_\ell}) &= \int_A \mu_y^1(F_{\varepsilon_m}\triangle F_{\varepsilon_{\ell}})d\mu_x^{n}(y)\\
        &= \int_A \mu_y^1\left(\left(\bigcup_{k<\omega}T^{-1}[F^{\varepsilon_m}_k]\cap C_k\right)\triangle\left(\bigcup_{k<\omega}T^{-1}[F^{\varepsilon_\ell}_k]\cap C_k\right)\right)d\mu_x^{n}(y)\\
        &= \int_A \mu_y^1\left(T^{-1}\{y\}\cap \left(\left(\bigcup_{k<\omega}T^{-1}[F^{\varepsilon_m}_k]\cap C_k\right)\triangle\left(\bigcup_{k<\omega}T^{-1}[F^{\varepsilon_\ell}_k]\cap C_k\right)\right)\right)d\mu_x^n(y)\\
        &= \int_A \mu_y^1(C_k\triangle C_{k'})d\mu_x^n(y)\quad (\textmd{where }\mu_y^1(E_y\triangle C_k)<\varepsilon_m,\textmd{ }\mu^y_1(E_y\triangle C_{k'})<\varepsilon_\ell)\\
        &\leq \int_A \mu_y^1(C_k\triangle E_y) + \mu_y^1(E_y\triangle C_{k'})d\mu_x^n(y)\\
        &\leq \int_A (\varepsilon_m + \varepsilon_\ell) d\mu_x^n(y)\\
        &\leq \varepsilon_m+\varepsilon_\ell
    \end{align*}
    Let $ F $ be the limit in $ \malg{\mu_x^{n+1}} $. We claim that this $ F $ satisfies the required property described. Indeed, 
    \[
    \mu_x^{n+1}(F) = \int_A \mu_y^1(F)d\mu_x^n(y)
    \]
    Since $ F_{\varepsilon_n}\rightarrow F $ in $ \malg{\mu_x^{n+1}} $, $ \mu_x^{n+1}(F_{\varepsilon_n}\triangle F)\rightarrow 0 $. Hence, for $ \mu_x^{n+1} $-a.e. $ y\in A $, $ \mu^1_y(F\triangle F_{\varepsilon_n}) \rightarrow 0 $ as well. Then for those $ y $, and for each $ n $, since $ \mu_y^1(F_{\varepsilon_n}\triangle E_y) = \mu_y^1(C_n\triangle E_y) <\varepsilon_n $, we see that $ \mu_y^1(F\triangle E_y) = 0 $, as desired. In particular,
    \[
    \mu_x^{n+1}(F) = \int_A \mu_y^1(F)d\mu_x^n(y) = \int_A \mu_y^1(E_y)d\mu_x^n(y)
    \]
    The statement about $ f$ follows immediately from measure disintegration.
\end{proof}
\section{Invariance}
\noindent Let $ f\in L^1(X,\mathcal{B},\mu) $, and recall the definitions for $ f^* $ and $ f_* $:
\begin{definition}
    The \textbf{average of $ f $ over $ E_x $} is
\[
A_{f,E_x} = \frac{\sum_{i=0}^{\ell(E_x)} \int_{E^i}f(z)d\mu^i_x(z)}{m_x(E_x)}
\]
Then, we let
\[
f^*(x) = \limsup_{m_x(E_x)\nearrow \infty} A_{f,E_x},\quad f_*(x) = \liminf_{m_x(E_x)\nearrow \infty} A_{f,E_x}
\]
where $ E_x $ varies over the (measurable) set of coherent trees, $ \CT_x $.
\end{definition}
\begin{lemma}
    The functions $ f^* $ and $ f_* $ are measurable $ \sigma(\Sigma^1_1) $-measurable. 
\end{lemma}
\begin{proof}
    We work with $ f^* $ as the proof of $ f_* $ is similar. Let $ \alpha\in \mathbb{Q} $. We will show that $ \{x\in X:f^*(x)>\alpha\} $ is $ \sigma(\Sigma^1_1) $. Observe that
    \[
    f^*(x)>\alpha \Longleftrightarrow \forall n\in \mathbb{N},\exists E\in \CT_x,(m_x(E_x)>n\wedge A_{f,E_x}>\alpha) 
    \]
    Consider the following sets $ \mathcal{A}_n\subseteq X\times \mathcal{Y} $ where $ (z,E)\in \mathcal{A}_n $ if and only if each of the following is true:
    \begin{enumerate}
        \item $ E\in \CT_z $
        \item $ m_{\pi(E)}(E) > n $
        \item $ A_{f,E} > \alpha $
    \end{enumerate}
    Then $ \mathcal{A}_n $ is Borel. To see this, we show each condition is Borel.\\\\
    To see condition $ (1) $ is Borel, we consider the function $ \id\times \pi:X\times \mathcal{Y}\rightarrow X\times X $, which is Borel. The set of $ (z,E) $ for which $ \pi(E) = z $, is Borel as it is $ (\id\times \pi)^{-1} [\Delta] $ (where $ \Delta $ is the diagonal). Intersecting with $ X\times \CT $ completes this part.\\\\
    For conditions $ (2) $ and $ (3) $, the relevant functions were shown to be Borel measurable in \ref{BorelCT}.\\\\
    Then let $ \mathcal{A} = \bigcap_{n\in \N} \mathcal{A}_n \subseteq X\times \mathcal{Y} $, and is also Borel. Then note that
    \[
    \{x:f^*(x)>\alpha\} = \proj_X(\mathcal{A})
    \]
    and thus is analytic.
\end{proof}
\begin{lemma}\label{importantforinvariance}
    For every $ y\in X $ and $ \mu^1_y $-almost every $ z\in T^{-1}\{y\} $, $ f^\star(y) \geq f^\star (z) $.
\end{lemma} 
\begin{proof}
    Suppose not. Fix $ y $, and suppose for some $ \eta,\varepsilon>0 $, that for some (Borel) measurable $ A\subseteq T^{-1}\{y\} $ with $ \mu^1_y (A) = \varepsilon $, that for every $ z\in A $,
    \[
    f^*(y) + \eta < f^*(z)
    \]
    Since $ \mu^1_y $ is a probability measure, by neglecting a $ \mu^1_y $-null set, we may assume $ A $ is Borel. We may further assume that $ f(y) = 0 $ after normalizing.\\\\
    Because $ f^*(y)+\eta < f^*(z) $ for each $ z\in A $, for those points, there are trees of arbitrarily large mass that witness the averages being larger than $ f^*(y) + \frac{\eta}{2} $. Hence, we first fix $ C >0 $ which is a lower bound for the mass of said witnesses. This $ C $ will be elaborated later. \\\\
    With the following claim, we now restrict our attention to witnesses of uniformly bounded length.
    \begin{claim}
        For each $ n\in \mathbb{N} $, let
        \[
        A_n:=\{z\in A:\exists E_z\in \CT_z,\left(\ell(E_z)\leq n\wedge A_{f,E_z} > f^*(y) + \frac{\eta}{2}\wedge m_x(E_x) > C\right)\}
        \]
        Then each $ A_n $ is measurable and there exists some natural number $ n $ for which $ A_n $ is positive $ \mu^1_y $-measure.
    \end{claim}
    \noindent To summarize, $ A_n $ the set of points in $ A $ whose witness of weight $ > C $ is also of length at most $ n $.
    \begin{proof}
        (of claim) Fix $ n\in \mathbb{N} $, and let $ \mathcal{A}_n \subseteq X\times \mathcal{Y} $ be defined by $ (z,E)\in \mathcal{A}_n $ if and only if each of the following are true:
        \begin{enumerate}
            \item $ z\in A $
            \item $ E\in \CT_z $
            \item $ A_{f,E} >f^*(y) + \frac{\eta}{2} $
            \item $ \ell(E) \leq n $
            \item $ m_{\pi(E)} (E) > C$
        \end{enumerate}
        As before, conditions (1), (2), and (3)  are Borel. For last condition (4), 
        \[
        \{E\in \CT_z:\forall k\geq n,\mu^k_z(E^k) = 0\} = \bigcap_{k\geq n} \{E\in \CT_z: \mu^k_z(E^k) = 0\}
        \]
        and the function $ E\mapsto \mu^k_z(E^k) $ is Borel by \ref{cor1}.\\\\
        For condition (5), this follows from \ref{BorelCT}.\\\\
        Finally, $ A_n $ is the projection of $ \mathcal{A}_n $ onto the first coordinate and thus is analytic, hence $ \sigma(\Sigma^1_1) $-measurable.\\\\
        Since $ A = \bigcup_{n\in \mathbb{N}} A_n $ is positive $ \mu^1_y $-measure and this is an increasing union, by continuity of measure, there is an $ n $ such that $ \mu^1_y(A_n) > 0 $.
    \end{proof}
    \noindent We may restrict our attention to $ A_n $ obtained from the claim, and we relabel $ A $ with this $ A_n $. Hence, we may further suppose that for every $ z\in A $, there is a witnessing tree of length at most $ n $, $ E_z:=E^0_z \sqcup \ldots \sqcup E^n_z $ such that $ m_z(E_z) \geq C $ and also witnesses
    \[
    A_{f,E_z} > f^*(y)+\frac{\eta}{2}
    \]
    From here, since for each $ z\in A $, there exists a desired witness, we now show there is a measurable assignment of said $ z $ to their witnesses.
    \begin{claim}\label{invariance_claim}
        There exists a $ \sigma(\Sigma^1_1) $-measurable map $ z\mapsto E_z $ of witnessing trees.
    \end{claim}
    \begin{proof}
        Recall that $ \mathcal{A}_n\subseteq X\times \mathcal{Y} $ from the previous claim was Borel, and for each $ z\in A $, the $ z $-slice of $ \mathcal{A}_n $ is non-empty. By Jankov, von Neumann \ref{JVN}, there is a $ \sigma(\Sigma^1_1) $ uniformization $ z\mapsto E_z $.
    \end{proof}
    \noindent Then for each $ i\leq n $, the map $ z\mapsto E_z^i $ is $ \sigma(\Sigma^1_1) $-measurable. We let $ E^i $ be the measurable witness obtained from using the gluing lemma, $ \ref{gluing}  $, on the maps $ z\mapsto E_z^i $.\\\\ 
    Let $ E_y = \{y\}\sqcup A \sqcup E^1\sqcup\ldots \sqcup E^n $, then $ S $ is a coherent tree. Then from \ref{coherencelemma}, observe that for each $ 1\leq i \leq n $,
    \begin{align*}
        \mu^{i+1}_y(E^i) &= \int \mu^{i}_z(E^i)d\mu^1_y(z)
    \end{align*}
    Hence,
    \begin{align*}
        m_y(E_y) &= \mu^0_y(\{y\}) + \mu^1_y(A) + \int \sum_{i=1}^n \mu^i_z(E^i)d\mu^1_y(z)\\
        &= \mu^0_y(\{y\}) + \int \sum_{i=0}^n \mu^i_z(E^i)d\mu^1_y(z)\\
        &= \mu^0_y(\{y\}) + \int \sum_{i=0}^n \mu^i_z(E^i_z)d\mu^1_y(z)\\
        &= \mu^0_y(\{y\}) + \int m_z(E_z)d\mu^1_y(z)
    \end{align*}
    At the same time, for each $ 1\leq i \leq n $,
    \[
    \int_{E^i}f(z)d\mu^{i+1}_y(z) = \int \int_{E^i_z}f(w)d\mu^{i}_z(w)d\mu^1_y(z)
    \]
    Thus,
    \begin{align*}
        \int_{\{y\}}f(z)d\mu^0_y(z) + \int_A f(w)d\mu^1_y(w) + \sum_{i=1}^{n}\int_{E^i}f(w)d\mu^{i+1}_y(w) &= f(y) + \int_A f(w)d\mu^1_y(w) \\ &\quad + \sum_{i=1}^n \int \int_{E^i_z}f(w)d\mu^i_z(w)d\mu^1_y(z)\\
        &= f(y) + \int_A f(z)d\mu^1_y(z) \\&\quad + \int \sum_{i=1}^n  \int_{E^i_z}f(w)d\mu^i_z(w)d\mu^1_y(z)\\
        &= f(y) + \int_A \sum_{i=0}^n \int_{E^i_z}f(w)d\mu^i_z(w)d\mu^1_y(z)
    \end{align*}
    Combining the above two, we have
\begin{align*}
    \frac{f(y) + \int_A \sum_{i=0}^n \int_{E^i_z}f(w)d\mu^i_z(w)d\mu^1_y(z)}{1 + \int_A m_z(E_z)d\mu^1_y(z)} &\geq \frac{f(y) + \int_A (f^*(y)+\frac{\eta}{2})\cdot m_z(E_z)d\mu^1_y(z)}{1+\int_A m_z(E_z)d\mu^1_y(z)}\\
    &= \frac{f(y) + (f^*(y)+\frac{\eta}{2})\cdot \int_A m_z(E_z)d\mu^1_y(z)}{1+\int_A m_z(E_z)d\mu^1_y(z)}\\
    &= (f^*(y)+\frac{\eta}{2})\cdot\frac{ \int_A m_z(E_z)d\mu^1_y(z)}{1+\int_A m_z(E_z)d\mu^1_y(z)}\\
\end{align*}
Finally, as $ \lim_{t\to\infty} \frac{t}{1+t} = 1 $, we can take $ m_z(E_z) $ is sufficiently large in \ref{invariance_claim} so that the above quantity is at least $ f^*(y) + \frac{\eta}{4} $.
\end{proof}
\begin{lemma}\label{T(A)=A}
    If $ \mu(A)>0 $ and $ T[A]\subseteq A $, then $ [A]_T = A $ off a null set.
\end{lemma}
\begin{proof}
    Let $ V = [A]_T\setminus A $. Then $ V $ is $ \mu $-nowhere $ T $-recurrent: i.e. for $ \mu $ almost every $ x\in V $, $ \{Tx,T^2x,\ldots\}\cap V = \emptyset $.\\\\
    Indeed, if $ x \in V$, then there are only finitely many $ n $ such that $ T^n (x)\in V $. To see this, note that for some $ a,b\in \N $, and $ y\in A $,
    \[
    T^a x = T^b y\in A 
    \]
    Then for every $ k>a $, $ T^k x\notin V $ as $ T[A]\subseteq A $. Thus, $ V $ is null.
\end{proof}
\begin{lemma}
    Let $ A = \{x\in X: f^*(x)\leq f^*(Tx)\} $. Then $ \mu(A) = 1 $.
\end{lemma}
\begin{proof}
    By measure disintegration, 
    \begin{align*}
        \mu(A) &= \int \mu^1_y (A)d\mu(y)\\
        &= \int \mu^1_{Tx}(A) d\mu(x)
    \end{align*}
    But by \ref{importantforinvariance}, for almost every $ z\in T^{-1}[\{Tx\}] $, $ f^*(z)\leq f^*(Tx) $. So $ A $ is co-null in $ T^{-1}[\{Tx\}] $, so $ \mu^1_{Tx}(A) = 1 $.
\end{proof}
\begin{proposition}\label{invariance}
    The function $ f^*$ is $ T $-invariant almost everywhere. That is, for almost every $ x $, $ f^*(x) = f^*(T(x)) $.
\end{proposition}
\begin{proof}
    Note that for each $ a\in \Q $, then $ X_{a }:=\{x\in X:f^*(x)\geq a\} $ is $ T $-invariant modulo a null set. Indeed, $ T[X_{a}]\subseteq X_a $ (a.e.) since if $ x\in X_a $, then $ f^*(Tx)\geq f^*(x) $ since $ A $ is conull, so $ Tx\in X_a $.
\end{proof}
\section{Local-Global Bridge}
To simplify some notation in this section, given a measure preserving transformation $ T $ and $ f\in L^1(X,\mathcal{B},\mu) $, let 
\[
P_T(f)(x) = \int_{T^{-1}\{x\}} f(y)d\mu^1_x(y) 
\] 
Recall the following property of measure disintegration, from \ref{coherencelemma},
\[
\int \int_{T^{-1}\{x\}}f(y)d\mu^1_x(y)d\mu(x) = \int f d\mu(x)
\]
We will use this to streamline a handful of useful calculations.
\begin{lemma}
    \[
    P_{T^2}(f)(x) = P_T(P_T(f))(x) 
    \]
\end{lemma}
\begin{proof}
    We see
    \[
    P_{T^2}(f)(x) = \int_{T^{-2}\{x\}} f(z)d\mu^2_x(z)
    \]
    At the same time,
    \begin{align*}
        P_T(P_T(f))(x) &= \int_{T^{-1}\{x\}} P_T(f)(y)d\mu^1_x(y)\\
        &= \int_{T^{-1}\{x\}}\int_{T^{-1}\{y\}} f(z)d\mu^1_y(z)d\mu^1_x(y)
    \end{align*}
    But, by \ref{coherencelemma}, the two quantities are indeed equal.
\end{proof}
\noindent An inductive argument gives 
\begin{equation}
    P_{T^n}(f) = P_T(P_{T^{n-1}}(f))
\end{equation}
and we observe
\begin{equation} 
    \int P_{T^n}(f)(x)d\mu(x) = \int fd\mu
\end{equation}
\begin{lemma}
    For any $ f\in L^1(X,\B,\mu) $,
    \[
    \norm{P_T(f)}_1 \leq \norm{f}_1
    \]
\end{lemma}
\begin{proof}
    We see by measure disintegration,
    \begin{align*}
        \int_X \abs{P_T(f)(x)}d\mu(x) &= \int_X \abs{\int_{T^{-1}\{x\}} f(y)d\mu^1_x(y)}d\mu(x)\\
        &\leq \int_X \int_{T^{-1}\{x\}} \abs{f(y)}d\mu_x^1(y) d\mu(x)\\
        & = \int_X |f(x)|d\mu(x)\\
        &= \norm{f}_1
    \end{align*}
    as desired.
\end{proof}
\begin{lemma}
    For every $ f\in L^1(X,\B,\mu) $, and $ n\in \N $, and almost every $ x\in X $,
    \[
    \sum_{i=0}^n \int_{T^{-i}\{x\}}\abs{f(y)}d\mu^i_x(y) <\infty
    \]
\end{lemma}
\begin{proof}
    Note that for each $ i $,
    \[
    \int_{T^{-i}\{x\}}\abs{f(y)}d\mu^i_x = P_{T_i}(\abs{f})(x)
    \]
    and the above is in $ L^1 $, and hence finite a.e.
\end{proof}
\begin{lemma}\label{localglobal}
    Let $ N\in \N $, with $ f\in L^1(X,\B,\mu) $. Then
    \[
    \int_X fd\mu = \int \frac{1}{N+1}\sum_{i=0}^N \int_{T^{-i}\{x\}} f(y)d\mu^i_x(y) d\mu(x)
    \]
\end{lemma}
\begin{proof}
    \begin{align*}
        \int \frac{1}{N+1}\int_{\bigsqcup_{i=0}^N T^{-i}\{x\}} f(y)d\mu^i_x(y) d\mu(x)&= \int \frac{1}{N+1} \int_{T^{-i}\{x\}} f(y)d\mu^i_x(y)d\mu(x)\\
        &= \int_X \frac{1}{N+1}\sum_{i=0}^N P_{T^{i}}(f)(x)d\mu(x)\\
        &= \frac{1}{N+1}\sum_{i=0}^N \int P_{T^i}(f)(x)d\mu(x)\\
        &= \int fd\mu(x)
    \end{align*}
    as desired.
\end{proof}
\begin{corollary}
    With the same setting as above, if $ A $ is $ T $-invariant, then
    \[
    \int_A f d\mu = \int_A \frac{1}{N+1} \sum_{i=0}^N \int_{T^{-1}\{x\}} f(y)d\mu_x^i(y)d\mu(x) 
    \]
\end{corollary}
\begin{corollary}\label{Lp-contraction}
    For any $ f\in L^p(X,\mathcal{B},\mu) $, 
    \[
    \norm{P_T(f)}_p \leq \norm{f}_p
    \]
\end{corollary}
\begin{proof}
    \begin{align*}
        \int_X \abs{P_T(f)(x)}^p d\mu(x) = \int_X \abs{\int_{T^{-1}(x)} f(y)d\mu_x^1(y)}^p d\mu(x)
    \end{align*}
    By Jensen's inequality, we have that
    \[
    \int_X \abs{\int_{T^{-1}(x)} f(y)d\mu_x^1(y)}^p d\mu(x) \leq \int_X \int_{T^{-1}(x)} \abs{f(y)}^pd\mu_x^1(y)d\mu(x) = \int_X \abs{f(x)}^pd\mu(x)
    \]
    as desired.
\end{proof}
\section{Main Theorems}
\subsection{Proof of the Backward Ergodic Theorem}
\begin{theorem}
    \label{The Theorem}(Backward Ergodic Theorem for arbitrary measure preserving transformations) Let $ X $ be a standard probability space, let $ T:X\rightarrow X $ be a measure-preserving transformation. Then for every $ f\in L^1(X) $, and for almost every $x \in X $, $ \abs{A_{f,E^x}} <\infty $ and the limit
    \[
    \lim_{m_x(E_x)\nearrow \infty} A_{f,E_x} 
    \]
    exists and equals $ \overline{f}= \mathbb{E}(f|\mathcal{B}_{\mathcal{I}}) $, the conditional expectation of $ f $ with respect to the $ \sigma $-algebra of $ T $-invariant subsets, where $ E_x$ varies over the set of coherent trees rooted at $x $.
\end{theorem}
\begin{proof} 
    We will replace $ X $ with the Borel $ X_0 $ established in \ref{universeofinterest}. Then, by subtracting $ \overline{f} $ from $ f $, we may assume $ \overline{f} = 0 $. We showed earlier in \ref{invariance} that $ f^* $ and $ f_* $ (the limsup and liminf respective) were $ T $-invariant almost everywhere. As such, it suffices to show that $ f^* = f_* = 0 $. To this end, it suffices to show that $ f^*\leq 0 $, and the proof for $ f_*\geq 0 $ is similar.\\\\
    For the sake of contradiction, suppose $ f^*>0 $ on some positive measure ($ T $-invariant) set. By restricting our attention to this set, we may assume that $ f^*>0 $ on all of $ X $. Let $ g = \min\{\frac{f^*}{2},1\} $, then $ g $ is bounded and integrable. Let $ c = \int gd\mu>0 $. As with \ref{importantforinvariance}, Jankov, von Neumann Uniformization ensures there is a measurable map $ X\to \mathcal{Y} $, $ x\mapsto E_x $ for which $ A_{f-g,E_x} > 0 $ for each $ x $. Then let $ \delta>0 $ be small such that whenever $ B\subseteq X $ has measure less than $ \delta $, then $ \int_B f-g d\mu >-\frac{c}{2} $. Let $ M $ be large so that $C:=f^{-1}[-M,\infty) $ has measure at least $ 1-\delta $. 
    \begin{claim}\label{tiling}
        For any $ \varepsilon>0 $, there exists $ N\in \N $ and subset $ X'\subset X $ of measure $ \geq 1-\varepsilon $ such that for all $ x\in X' $, there exists a coherent subtree $ S $ of complete tree $ \triangleright_T^N\cdot x $ so that $ m_x(S')\geq (1-\varepsilon)m_x(\triangleright_T^N\cdot x) $ and with $ A_{f-g,S} >0 $. 
    \end{claim}
    \begin{proof}
        Let $ \varepsilon>0 $, and let $ L $ be large such that 
        \[
        \mu(\{x\in X:\ell(E_x)\geq L\}) < \frac{\varepsilon^2}{2}
        \]
        Note that this set is indeed measurable since the map $ x\mapsto E_x $ is measurable, and the map $ \ell:\mathcal{Y}\rightarrow \mathbb{N} $ where $ \ell(E_x) $ is the length of $ x $ is Borel.\\\\
        Let $ D = \{x\in X:\ell(E_x)\geq L\}$. Then, let $ N $ be large enough so that $ \frac{L}{N}<\frac{\varepsilon}{2} $. \\\\
        Then the set
        \[
        X':=\{x\in X:A_{1_D,\triangleright^N_T\cdot x}\geq \frac{\varepsilon}{2}\}
        \]
        has measure less than $ \varepsilon $. To see this, 
        \begin{align*}
            \mu(D) &= \int_X 1_{D} d\mu\\
            &= \int_X A_{1_{D},\triangleright^N_T\cdot x}d\mu\\
            &\geq \int_{X'} A_{1_{D},\triangleright^N_T\cdot x} d\mu \\
            &\geq \mu(X')\cdot \frac{\varepsilon}{2}
        \end{align*}
        Since $ \mu(D) <\frac{\varepsilon^2}{2} $, $ \mu(X') < \varepsilon $, as desired.\\\\
        We claim $ X' $ is the desired set. We proceed by an inductive `greedy' argument on $ n\leq N-L $. 
        \begin{itemize}
            \item Let $ S^0:= E_x $. Let $ Q_0 = T^{-1}\{x\}\setminus E_x^1 $. Then $ Q_0 $ is Borel. \\\\
            Here, we are considering the portion of $ T^{-1}\{x\} $ not already covered by $ S^0 $, and in the next step, we will attach witnesses for those points to have $ T^{-1}\{x\}$ be entirely included in the final $ S $.
            \item For each $ y \in Q_0 $, we have $ E_y $, each disjoint from each other, and also disjoint from $ E^x $. Let $ S^1 = S^0\cup \bigcup_{y\in Q_0}E_y $. \\\\
            By the gluing lemma, \ref{gluing}, $ S^1 $ is still a coherent measurable tree. Let $ Q_1 = T^{-2}\{x\}\setminus S^1_2 $, which is also measurable.\\\\
            Again, we consider the portion of $ T^{-2}\{x\} $ not already covered by $ S^1 $, and we repeat the process inductively.
            \item Given $ Q_n $ measurable, and $ S^n $ a measurable coherent tree, let $ S^{n+1} = S^n\cup \bigcup_{y\in Q_n}E_y $ and let $ Q_{n+1} = T^{-(n+1)}\{x\}\setminus S^{n+1}_{n+2} $. Again, by the Gluing lemma, \ref{gluing}, $ S^{n+1} $ and $ Q_{n+1} $ are all measurable.
        \end{itemize}
        Let $ S $ be the coherent tree obtained after this process, which is measurable. Furthermore, we have at most $ \triangleright^n_T\cdot x\setminus \triangleright_T^{N-L}\cdot x $ untiled, as desired.
    \end{proof}
    \begin{claim}\label{touse}
        For each $ x\in X' $ above,
        \[
        A_{1_C\cdot (f-g),\triangleright_T^n\cdot x} \geq -(M+1)\varepsilon
        \]
    \end{claim}
    \begin{proof}
        From the above construction of $ X' $, on a measurable subtree $ S\subseteq \triangleright^n_T \cdot x $ that is at least $ 1-\varepsilon $ of the mass of $ \triangleright_n^T\cdot x $, the average of $ f-g $ is positive. Thus, $ (1|_D\cdot f-g)|_{X'} $ is non-negative. But on the rest of $ \triangleright_T^n\cdot x $, the function $ 1_C\cdot (f-g) $ is at least $ -(M+1) $ by the definition of $ C $. Hence, the weighted average of $ 1_C\cdot (f-g) $ on all of $ \triangleright_T^n\cdot x $ is at least $ -(M+1)\cdot \varepsilon $.
    \end{proof}
    \noindent Combining Claim \ref{touse} and Lemma \ref{localglobal}, we see
    \begin{align*}
        \int_C (f-g)d\mu &= \int_X A_{1_C\cdot (f-g),\triangleright_T^n\cdot x} d\mu(x)\\
        &= \int_{X'} A_{1_C\cdot (f-g),\triangleright_T^n\cdot x} d\mu(x) + \int_{X\setminus X'} A_{1_C\cdot (f-g),\triangleright_T^n\cdot x} d\mu(x)\\
        &\geq -(M+1)\varepsilon - (M+1)\varepsilon \\
        &= -2(M+1)\varepsilon \\
        &= -\frac{c}{3} 
    \end{align*}
    But then,
    \begin{align*}
        0 &= \int_X \overline{f} d\mu \\
        &= \int_X fd\mu\\
        &= c + \int_X (f-g)d\mu\\
        &= c + \int_C (f-g)d\mu + \int_{X\setminus C} (f-g)d\mu\\
        &> c-\frac{c}{3}-\frac{c}{3}\\
        &> 0
    \end{align*}
    a contradiction.
\end{proof}
\subsection{Proof of the Backward Maximal Ergodic Theorem}
\noindent We also prove a Backward Ergodic Maximal Theorem for the uncountable case, as stated below: 
\begin{theorem}
    (Backward Ergodic Maximal Theorem) Let $ T $ be as in \ref{The Theorem}. Then for every $ f\in L^1(X,\mu) $, if we define $ f^\star  = \sup_{E_x\in \CT_x} A_{f,E_x} $ for each $ x\in X $. Then for any $ \lambda\in \mathbb{R} $, 
    \[
    \int_{\{f^\star>\lambda\}} f d\mu \geq \lambda \mu(\{f^\star>\lambda\})
    \]
\end{theorem}
\begin{proof}
    First, we see that $ f^\star $ is $ \sigma(\Sigma^1_1) $-measurable. To see this, let $ \alpha\in \mathbb{R} $. We will show that $ \{f^\star>\alpha\} $ is analytic. Consider $ \mathcal{A}\subseteq X\times \mathcal{Y} $ defined as follows: $ (x,E)\in \mathcal{A} $ if and only if:
    \begin{enumerate}
        \item $ E\in \CT_x $
        \item $ A_{f,E}>\alpha $
    \end{enumerate}
    As we have seen earlier, the two conditions defining $ \mathcal{A} $ are Borel. Hence, $ \{f^\star>\alpha\} = \proj_X(\mathcal{A}) $, and is thus analytic.\\\\
    Now, let $ \lambda\in \mathbb{R} $. Let $ Y = \{x\in X:f^\star>\lambda\} $. We want to show that
    \[
    \int_Y fd\mu \geq \lambda \mu(Y) 
    \]
    It suffices to show that for every $ \varepsilon>0 $, $\int_Y(f-\lambda)d\mu \geq -\varepsilon $.
    \begin{definition}\label{def for min}
        Let $ x\in Y $. Then let $ E_x $ be a coherent tree that witnesses $ x\in Y $. We say $ E_x $ is \textbf{minimal} if there is no proper \textbf{subtree}, $ F_x $ such that $ A_{f,F_x} > \lambda $. \\\\
        More precisely, $ F_x $ is a \textbf{subtree} of $ E_x $ if for each $ i $,
        \[
        \mu^i_x(F_x^i\setminus E_x^i) = 0
        \]
        And, if $ F_x $ is a subtree of $ E_x $ such that $ A_{f,F_x}>\lambda $, then $ E_x $ and $ F_x $ agree ($\mu_x^i$)-almost-everywhere at each level. In particular, $ m_x(F_x) = m_x(E_x) $.
    \end{definition}
    \begin{claim}\label{minimal gives right thing}
        Suppose $ x\in Y $, and $ E_x $ is a minimal witness as in \ref{def for min}. Then for $ m_x $-almost-every $ y\in E_x $, $ y\in Y $. More precisely, for each $ 0\leq i \leq \ell(E_x) $, 
        \[
        \mu_x^i(E_x^i\setminus Y) = 0
        \]
    \end{claim}
    \noindent We call the property of $ E_x $ in the hypothesis of \ref{minimal gives right thing}, \textbf{$Y$-saturated}. 
    \begin{proof}
        (of claim) Suppose, for the sake of contradiction, that for some $ 0\leq i \leq \ell(E_x) $, $ \mu^i_x(E_x^i\setminus Y) > 0 $. Then we can simply remove the portion of $ E_x $ behind that subset and obtain a proper subtree of strictly smaller mass, on which the average of $ f $ is still greater than $ \lambda $, contradicting minimality of $ E_x $.
    \end{proof}
    \begin{claim}
        For every $ x\in Y $, there exists a minimal witnessing tree $ E_x $.
    \end{claim}
    \begin{proof}
        (of claim) Let $ x\in Y $. Since $ \sup_{E\in \CT_x} A_{f,E} > \lambda $, let $ E\in \CT  $ be a witness. We perform a transfinite induction on countable ordinals:
        \begin{itemize}
            \item Let $ \mathcal{E}_0 = E $.
            \item Suppose $ i<\omega_1 $, and we have $\mathcal{E}_i \in \CT_x $ witnessing $ A_{f,\mathcal{E}} >\lambda $. Let $ \mathcal{E}_{i+1} $ be a proper subtree of $ \mathcal{E}_i $ with $ m_x(\mathcal{E}_{i+1})\leq m_x(\mathcal{E}_i) $. 
            \item If $ i<\omega_1 $ is a limit ordinal, let $ \mathcal{E}_i = \bigcap_{j<i}\mathcal{E}_i $ (where we do coordinate-wise intersection). Then $ \mathcal{E}_i \in \CT_x $, and $ m_x(\mathcal{E}_i)\leq m_x(\mathcal{E}_j) $ for each $ j<i $.
        \end{itemize}
        Then the function $ \zeta:\omega_1\to [1,\infty) $, $ i\mapsto m_x(\mathcal{E}_i) $ must stabilize at some countable ordinal $ i $. Then $ \mathcal{E}_i $ is the minimal witness.
    \end{proof}
    \begin{claim}\label{maximaluniformization}
        There is a $ \sigma(\Sigma^1_1) $-measurable function $ \gamma:X\to \mathcal{Y} $ such that for every $ x\in Y $, $ \gamma(x) $ is $ Y $-saturated. 
    \end{claim}
    \begin{proof}
        (of claim) Define $ \mathcal{A}\subseteq X\times \mathcal{Y} $ where $ (x,E)\in \mathcal{A} $ if and only if the following are true:
        \begin{enumerate}
            \item $ E\in \CT_x $
            \item $ A_{f,E}>\lambda $
            \item $ E $ is $ Y $-saturated. 
        \end{enumerate}
        We have already established in earlier proofs that conditions (1) and (2) are Borel. And it's easy to see that condition (3) is also Borel.\\\\
        Then by Jankov von-Neumann, \ref{JVN}, there exists a $ \sigma(\Sigma^1_1) $-uniformization $ \gamma:X\to \mathcal{Y} $, as desired.
    \end{proof}
    \noindent With the uniformization, we can proceed with a proof similar to that in \ref{The Theorem}. Let $ \varepsilon>0 $. As in \ref{The Theorem}, we can assume that $ f $ is bounded below. 
    \begin{claim}
        For any $ \varepsilon>0 $, there exists $ N\in \mathbb{N} $ and subset $ X'\subset X $ of measure $ \geq 1-\varepsilon $ such that for all $ x\in X' $, there exists a coherent subtree $ S $ of complete tree $ \triangleright^N_T\cdot x $ so that $ m_x(S)\geq (1-\varepsilon)m_x(\triangleright_T^N\cdot x) $ and with
        \[
        A_{f,S}>\lambda 
        \]
    \end{claim}
    \begin{proof}
        The proof of the claim follows identically to that of \ref{tiling} in which we use a greedy argument using the gluing lemma, \ref{gluing}, with the measurable uniformization obtained in \ref{maximaluniformization}.
    \end{proof}
    \noindent We further have, using the same argument as earlier,
    \begin{claim}
        For each $ x\in X' $ above,
        \[
        A_{1_Y\cdot(f-\lambda),\triangleright^N_T\cdot x} \geq -\frac{\varepsilon}{2}
        \]
    \end{claim}
    \noindent Finally, by \ref{localglobal}, we see
    \begin{align*}
        \int_Y (f-\lambda)d\mu &= \int A_{1_Y\cdot (f-\lambda),\triangleright_T^N\cdot x}d\mu\\
        &= \int_{X'} A_{1_Y\cdot (f-\lambda),\triangleright^N_T\cdot x} d\mu + \int_{X\setminus X'}A_{1_Y\cdot (f-\lambda),\triangleright^N_T\cdot x} d\mu \\
        &\geq \frac{\varepsilon}{2} + \int_{X\setminus X'} A_{1_Y\cdot (f-\lambda),\triangleright^N_T\cdot x}d\mu
    \end{align*}
    Taking $ X' $ to be sufficiently close measure to $ X $ and since $ f $ can be assumed to be bounded below, we can conclude $ \int_{Y}(f-\lambda)d\mu \geq -\varepsilon $, as desired.
\end{proof}
\subsection{Convergence in \texorpdfstring{$L^p$}{Lp}}
\noindent As in the countable case, we also have convergence in $ L^p $ if we average over complete trees. 
\begin{corollary}
    Let $ T $ be as in \ref{The Theorem}. Let $ 1\leq p < \infty $. If $ f\in L^p(X,\mu) $, let $ f'_N(x) = A_{\triangleright^N_T\cdot x, f} $. Then $ f'_N $ converges to $ \overline{f} $ for almost every $ x $, and also in $ L^p $.
\end{corollary}
\begin{proof}
    First note that since $ X $ is a probability space, $ f\in L^1 $. Furthermore, as $ N\to \infty $, $ m_x(\triangleright^N_T\cdot x) = N+1 $, which increases to $ \infty $, we have convergence of $ f'_N $ almost everywhere from \ref{The Theorem}. \\\\
    Now, note that if $ f\in L^\infty $, then $ \abs{f'_N}\leq \norm{f}_\infty $ almost everywhere, so by Dominated Convergence, $ f'_N\to \overline{f} $ in $ L^p $.\\\\
    Thus, suppose $ f $ is not bounded. Let $ \{f_k:k\in \omega\} $ be a sequence of $ L^\infty $ functions that converge to $ f $ in $ L^p $. Let $ \varepsilon>0 $. Fix $ k $ large enough so that $ \norm{f-f_k}<\frac{\varepsilon}{3} $. Then for every $ N\in \mathbb{N} $, note that
    \begin{align*}
        \norm{f'_N - \overline{f}}_p &\leq \norm{f'_N - (f_k)'_N}_p + \norm{(f_k)'_N - \overline{f_k}}_p + \norm{\overline{f_k} - \overline{f}}_p
    \end{align*}
    By \ref{Lp-contraction}, the first part of the sum is at most $ \frac{\varepsilon}{3} $, and by properties of conditional expectation, the third part of the sum is at most $ \frac{\varepsilon}{3} $ too. Finally, as $ f_k\in L^\infty $, let $ M $ be large such that for all $ N>M' $, $ \norm{(f_k)'_N - \overline{f_k}}_p <\frac{\varepsilon}{3}$. 
\end{proof}
\section{Appendix}
\subsection{Recurrence}
\noindent The purpose of this section is to provide details for the proof of Lemma \ref{T(A)=A}. 
\begin{definition}
    Set $ A\subseteq X $ is \textbf{$ T $-wandering} if for every $ m,n\in\mathbb{N} $, with $ m\neq n $, $ T^{-n}[A]\cap T^{-m}[A] $ are disjoint. 
\end{definition}
\noindent Clearly, any $ T $-wandering set is measure $ 0 $.
\begin{definition}
    Set $ A\subseteq X $ is \textbf{$ T $-recurrent} if for every $ x\in A $, $ \exists n\geq 1 $ so that $ T^n(x)\in A $. \\\\
    And $ A $ is \textbf{$ \mu $-nowhere $ T $-recurrent} every positive measure subset of $ A $ is not $ T $-recurrent.  
\end{definition}
\begin{proposition}
    Every positive measure set $ A $ is $ T $-recurrent almost everywhere.  
\end{proposition}
\begin{proof}
    Let $ A' = \{x\in A:\forall n\geq 1 ,(T^n(x)\notin A)\} $. Then $ A' $ is $ T $-wandering. Indeed, suppose $ n,m\geq 1 $ with $ n< m $, and suppose for the sake of contradiction that $ z\in T^{-n}[A']\cap T^{-m}[A'] $. Then $ T^n(z)\in A' $, and $ T^m(z)\in A' $. But $ T^m(z) = T^{m-n}(T^n(z)) \in A'$, a contradiction. \\\\
    Thus, since $ A' $ is $ T $-wandering, it is measure $ 0 $, so $ A\setminus A' $ is the desired witness to $ T $-recurrence.
\end{proof}
\subsection{Connections to the Countable Case}
\noindent Our result implies Tserunyan-Zomback's previous result in the countable case. This is because if $ T $ is countable-to-one, then the measure disintegration is indeed the Radon-Nikodym cocycle. More specifically, if $ \rho $ is the Radon-Nikyodym Cocycle for a countable-to-one $ T $, then for $ A\subseteq T^{-1}\{x\} $,
\[
\mu_x^1(A) = \sum_{y\in A} \rho_x(y)
\]
To see this, it suffices to check that the quantity given in the right-hand-side is, indeed, a disintegration.
\begin{proposition}
    Suppose $ T:X\rightarrow X $ is an aperiodic countable-to-one measure preserving transformation. Let $ x\mapsto \mu_x^i $ be corresponding measure disintegrations. Let $ \rho $ be its Radon-Nikodym cocyle, and for each $ x\in X $, $ i\in \N $, define $ \nu_i^x:\mathcal{B}\rightarrow [0,1] $ by
    \[
    \nu_x(A) = \sum_{y\in A\cap T^{-i}\{x\}} \rho_x(y)
    \]
    Then for every $ i $ and almost every $ x\in X $, $ \nu^i_x = \mu^i_x $.
\end{proposition}
\begin{proof}
    Let $ i\in \N $. It suffices to show that $ x\mapsto \nu^i_x $ is also a disintegration of $ T^i $. First, indeed $ \nu^i_x $ is a probability measure on $ X $. This follows from the Local-Global-Bridge in \cite{anushpaper}. We then check that it is a disintegration of $ T^i $.
    \begin{itemize}
        \item Clearly $ \nu^i_x $ is supported on $ T^{-i}\{x\} $.
        \item The observation that
        \[
        \int \int f(z)d\nu^i_x d\mu(x) = \int fd\mu
        \]
        follows from the Local-Global bridge in \cite{anushpaper}
        \item Finally, if $ f\in L^1(X,\mathcal{B},\mu) $, then the function
        \[
        x\mapsto \sum_{y\in T^{-i}\{x\}} f(y)\rho_x(y)
        \]
        is Borel as the Radon-Nikodym cocycle is Borel.
    \end{itemize}
    Thus, by measure disintegration, $ \nu^i_x = \mu^i_x $ for almost every $ x\in X $.
\end{proof}
\noindent We also check that the average of $ f $ over a finite-height tree rooted at $ x $, $ \tau_x $, is indeed the average computed here.
\begin{proposition}
    Suppose $ T:X\rightarrow X $ is an aperiodic countable-to-one measure preserving transformation and let $ f\in L^1(X,\mathcal{B},\mu) $. For each $ i $, let $ x\mapsto \mu^x_i $ be a measure disintegration. Let $ \rho $ be its Radon-Nikodym cocycle. Then for almost every $ x $ and every $ \tau_x $, a finite height tree rooted at $  x$, 
    \[
    A^\rho_f[\tau_x] = A_{f,\tau_x}
    \]
\end{proposition}
\begin{proof}
    Let $ \tau_x $ be a finite-height tree rooted at $ x $. Then
    \begin{align*}
        A^\rho_f[\tau_x] &= \frac{\sum_{y\in \tau_x} f(y)\rho_x(y)}{\rho_x(\tau_x)}\\
        &= \frac{\sum_{i=0}^{\ell(\tau_x)} \sum_{y\in \tau_x^i} f(y) \rho_x(y)}{\sum_{i=0}^{\ell(\tau_x)} \sum_{y\in \tau_x^i} \rho_x(y)}\\
        &= \frac{\sum_{i=0}^{\ell(\tau_x)} \int f d\nu_x^i}{\sum_{i=0}^{\ell(\tau_x)} \nu^i_x (\tau_x^i)}\\
        &= A_{f,\tau_x}
    \end{align*}
    as desired.
\end{proof}
\noindent The last thing to verify is that for each $ x $, the set of coherent trees, $ \CT_x $ coincides with the family of finite-height trees in the countable-to-one case. But this follows because for each $ i $, the set $ T^{-i}\{x\} $ is countable, and hence any subset is measurable. 
\newpage
\printbibliography
\end{document}